\documentclass[english,a4paper,10pt]{amsart}
\usepackage[latin1]{inputenc}
\usepackage{amssymb,amsthm}
\usepackage[all]{xy}
\usepackage{amsmath}
\usepackage{indentfirst}
\usepackage{amsfonts}
\usepackage{textcomp}
\usepackage{graphicx}
\usepackage{enumerate}
\usepackage{hyperref}

\newtheorem{lemma}{Lemma}[section]
\newtheorem{theorem}[lemma]{Theorem}

\newtheorem{proposition}[lemma]{Proposition}
\newtheorem{corollary}[lemma]{Corollary}
\theoremstyle{definition}
\newtheorem{definition}[lemma]{Definition}
\theoremstyle{remark}
\newtheorem{remark}[lemma]{Remark}

\newtheorem{example}[lemma]{Example}

\newcommand{\mult}{\mathrm{mult}}

\newcommand{\Q}{\mathbb{Q}}

\newcommand{\R}{\mathbb{R}}

\newcommand{\N}{\mathbb{N}}
\newcommand{\C}{\mathbb{C}}

\newcommand{\B}{\mathbf{B}}
\newcommand{\D}{\Delta}

\newcommand{\supp}{\mbox{Supp}}

\newcommand{\exc}{\mathrm{Exc}}

\newcommand{\ord}{\mathrm{ord}}
\newcommand{\num}{\textrm{num}}
\newcommand{\NNef}{\mathrm{NNef}}

\def\ord{\operatorname{ord}}

\def\supp{\text{Supp}}

\def\codim{\text{codim}}

\begin{document}

\title{Asymptotic base loci on singular varieties}
\thanks{{\it Math classification:} 14C20}
\thanks{\today}

\author{Salvatore Cacciola and Lorenzo Di Biagio}
\address{Dipartimento di Matematica, Universit\`a degli Studi ``Roma Tre''- Largo S.\ Leonardo Murialdo 1, 00146, Roma, Italy}
\email{cacciola@mat.uniroma3.it}
\address{Dipartimento di Matematica, Universit\`a degli Studi ``Roma Tre''- Largo S.\ Leonardo Murialdo 1, 00146, Roma, Italy}
\email{dibiagio@mat.uniroma3.it}

\begin{abstract}
We prove that the non-nef locus and the restricted base locus of a pseudoeffective divisor coincide on KLT pairs. We also extend to KLT pairs F.\ Russo's characterization of nef and abundant divisors by means of asymptotic multiplier ideals.\end{abstract}

\maketitle


\setcounter{section}{0}
\setcounter{lemma}{0}

\section{Introduction}
In the paper \cite{ELMNP} by Ein, Lazarsfeld, Musta\c t\u a, Nakamaye and Popa the asymptotic behavior of base loci of line bundles
on complex projective varieties is investigated by making use of various invariants.
In particular, when $X$ is a normal projective variety and $D$ is a big $\mathbb{Q}$-Cartier divisor on $X$,
they defined the \textit{restricted base locus} of $D$ as
$$\B_-(D):=\bigcup_A \B(D+A),$$
where the union is taken over all ample $\Q$-Cartier divisors on $X$ (see Definition \ref{bmenodef}).
In this way many pathologies associated to the mere stable base locus $\B(D)$ disappear.

In the same paper the five authors, inspired by the work of Nakayama in \cite{Nakayama}, defined also an asymptotic measure of the singularities of $D$: if
 $v$ is a geometric valuation on $X$, then the \textit{asymptotic order of vanishing} of $D$ along $v$ is
$$v(\|D\|):=\lim_{p\rightarrow \infty} \frac{v(|pD|)}{p}$$ (see Definition  \ref{asymptotic}).

They noticed that $v(\|D\|)$ is a numerical invariant and, by passing to limits, it is possible to define a \textit{numerical order of vanishing} on
every pseudoeffective $\R$-Cartier divisor on $X$ that, following the notation of \cite{BBP}, we will denote by $v_{num}(D)$
(see \S 2.2).
Again from \cite{BBP} we borrow the definition of the non-nef locus $\NNef(D)$ as the subset of $X$ given by the union of all the centers of
the valuations $v$ such that $v_{num}(D)>0$ (Definition \ref{pino}). This is a straightforward generalization of the numerical base locus $\mathrm{NBs}(D)$ of Matsuda (\cite{M}) and Nakayama (\cite{Nakayama}).
As the name itself suggests, $\NNef(D)=\emptyset$ if and only if $D$ is nef. Notice that the same holds for $\B_- (D)$.
Therefore it is natural to wonder if these two loci coincide in general.

By \cite[Lemma V.1.9(1)]{Nakayama} (see also \cite[Proposition 2.8]{ELMNP}) the following holds:

\begin{theorem}[Nakayama]\label{thm:casoliscio}
Let $X$ be a smooth projective variety and let $D$ be a big $\Q$-Cartier divisor on $X$.
Then $\B_-(D)=\NNef(D)$.
\end{theorem}

It is then trivial to show that, passing to limits,  the same equality holds for pseudoeffective $\R$-divisors.


In this paper we generalize this result to normal varieties with mild singularities:

\begin{theorem} \label{riassuntivo}
Let $X$ be a normal projective variety and suppose that there exists an effective $\Q$-Weil divisor $\D$ such that $(X,\D)$ is a KLT pair.

Then for every $\R$-Cartier pseudoeffective divisor $D$ on $X$ we have that $\B_-(D)=\NNef(D)$.
\end{theorem}

This is a partial answer to a conjecture of Boucksom, Broustet and Pacienza (see \cite[Conjecture 1.9]{BBP}).

The hypothesis of the existence of a KLT boundary is necessary in our proof because this is the only context where asymptotic multiplier ideals
are not strongly influenced by the singularities of $X$, so that they reflect the asymptotic behavior of the base loci.
We could avoid it only in the case of surfaces (Corollary \ref{superfici}). See also Corollary \ref{corb-ennef} and Corollary \ref{dim0} for slight generalizations.

On the other hand, one can consider \textit{asymptotic} orders of vanishing $v(\|D\|)$ for every effective divisor.
These are in general different from \textit{numerical} orders of vanishing
(Remark \ref{casobig}) and we have that $v(\|D\|)=0$ for every geometric valuation $v$ if and only if $D$ is nef and abundant
(Lemma \ref{Russo_valuations}).
In analogy with the definition of the non-nef locus, we use these asymptotic orders of vanishing to define a \emph{non nef-abundant locus} $\mathrm{NNA}(D)$
(see Definition \ref{gigetto}).
In particular, we prove, for every effective Cartier divisor $D$ on a normal projective variety $X$ admitting an effective KLT boundary $\D$, the equality
$$\mathrm{NNA}(D)=\bigcup_{p \in \N} \mathcal{Z}(\mathcal{J}((X,\D);\|pD\|))$$
(see Corollary \ref{cor:merging}).
Note that, when $D$ is big, this means
$$\B_-(D)=\bigcup_{p \in \N} \mathcal{Z}(\mathcal{J}((X,\D);\|pD\|)),$$
so that, in particular, we get a generalization of \cite[Corollary 2.10]{ELMNP}.

Moreover, as a corollary of this result, we give a characterization
of nef and abundant divisors in terms of triviality of asymptotic multiplier ideals
(see Corollary \ref{charnefabundant}), generalizing to the KLT case the main theorem of F.\ Russo's paper \cite{R}.\\

This paper is structured as follows:
in Section 2 we just review the relevant definitions and introduce the non nef-abundant locus;
in Section 3 we study the relationship between the restricted base locus and the non-nef locus in the case of surfaces;
in Section 4 we prove the main theorem; in Section 5 we present some consequences about nef and abundant divisors.

\subsection*{Acknowldegments}
We are deeply grateful to Prof.\ Angelo Felice Lopez for proposing us the problem and for many helpful discussions.
We also wish to thank Prof.\ Tommaso de Fernex for some useful conversations, Prof. S\'ebastien Boucksom for suggesting us a simpler proof of Proposition \ref{spariscedelta} and the anonymous referee for many valuable suggestions.

\section{Preliminaries}
\subsection{Notation and conventions}

We will work over the field of complex numbers $\mathbb{C}$. 
Given a variety $X$ and a coherent sheaf of ideals $\mathcal{J}\subseteq \mathcal{O}_X$ we denote by $\mathcal{Z}(\mathcal{J})$
the closed subset of $X$ defined by $\mathcal{J}$, without any scheme structure.

A \textit{pair} $(X,\D)$ consists of a normal projective variety $X$ and a Weil $\mathbb{Q}$-divisor $\D$ on $X$ such that
$K_X+\D$ is $\Q$-Cartier. A pair is \textit{effective} if $\Delta \geq 0$.
From now on, unless otherwise stated, by \emph{divisor} we mean an integral Cartier divisor;
for $\mathbb{K}=\Q$, $\R$, by \emph{$\mathbb{K}$-divisor} we mean a $\mathbb{K}$-Cartier divisor.
Given a divisor (or line bundle) $D$ on a variety $X$, we denote by $\kappa(X,D)$ its Kodaira dimension.

Given a smooth variety $X$ and a $\Q$-divisor $D$ on $X$, we denote by $\mult_x D$ the \textit{multiplicity} at $x \in X$ of $D$, in the sense of \cite[Definition 9.3.1]{LazII}.

\subsection{Multiplier ideals}

In this subsection we recall some definitions and some well-known facts about multiplier ideals.
We refer to \cite{LazII} for a more exhaustive treatment of this subject.


\begin{definition}(cf. \cite[Definition 9.3.56]{LazII}).
Let $(X,\D)$ be a pair and let us denote by $\mu:Y\to X$ a log-resolution of the pair $(X,\D)$.
Then for every prime divisor $E$ on $Y$,
there exist canonically defined rational numbers $a(E)=a(E,X,\D)$
such that
$$K_{Y}\equiv\mu^*(K_X+\D)+\sum a(E) E.$$
 Note that all but finitely many of these numbers are zero.

If $D$ is a $\Q$-divisor on $X$ and $\mu$ is also a log-resolution of $(X,\D+D)$, then
we can consider uniquely defined numbers $b(E)\in \Q$ such that
$\mu^*(-D)=\sum b(E)\cdot E$.
The \emph{multiplier ideal} associated to $D$ on the pair $(X,\D)$ is the sheaf
$$\mathcal{J}((X,\D);D):=\mu_*\mathcal{O}_Y\left(\sum \ulcorner a(E)+b(E)\urcorner E \right),$$ where $\ulcorner \cdot \urcorner$ denotes the round up.

By convention, we put $\mathcal{J}(X,\D):=\mathcal{J}((X,\D);0)$.

Similarly, if $|V|$ is a non-empty linear series on $X$ and $\mu$ is also a log-resolution of $|V|$,
we can write $\mu^*|V|=|W|+F$,
where $|W|$ has no base points and $F+\exc(\mu)$ has simple normal crossing support.
We can consider uniquely defined  numbers $d(E)\in \Q$ such that
$-F=\sum d(E)\cdot E$.
If $c>0$ is a rational number, the \emph{multiplier ideal} associated to $c$ and $|V|$ on the pair $(X,\D)$ is the sheaf

$$\mathcal{J}((X,\D);c|V|):=\mu_*\mathcal{O}_Y\left(\sum \ulcorner a(E)+c \cdot d(E)\urcorner E \right).$$
\end{definition}

Note that $\mathcal{J}((X,\D); c|V|)\subseteq \mathcal{O}_X$ if the pair $(X,\D)$ is effective
and $\mathcal{J}((X,\D); D)\subseteq \mathcal{O}_X$ if, in addition, the $\Q$-divisor $D$ is effective.

Moreover it is immediate to see that $\mathcal{J}((X,\D); D)=\mathcal{J}(X,\D+D)$.

\begin{definition}(cf. \cite[Definition 11.1.2]{LazII}).
Let $(X,\D)$ be a pair. Consider a divisor $D$ on $X$ such that $\kappa(X,D)\geq 0$ and a rational number $c>0$.
The \emph{asymptotic multiplier ideal sheaf} associated to $c$ and $D$ on the pair $(X,\D)$, denoted by
$$\mathcal{J}((X,\D); c\|D\|),$$
is defined as the unique maximal member in the family of ideals $\left\{ \mathcal{J}((X,\D);\frac{c}{p} |pD|)\right\}$, where $p$ runs over all positive integers such that $|pD| \not = \emptyset$.
\end{definition}

\begin{proposition}(cf. \cite[Proposition 9.2.26, p.\ 185, Proposition 11.1.4]{LazII}).
Let $(X,\D)$ be a pair and let $|D|$ be a complete linear series on $X$ such that $\kappa(X,D)\geq 0$.
Then for every sufficiently large and divisible $k\in \N$, if $D_k\in |kD|$ is a general divisor, we have that
$$\mathcal{J}((X,\D); \|D\|)=\mathcal{J}((X,\D); \frac{1}{k}|kD|)=\mathcal{J}((X,\D); \frac{1}{k}D_k).$$

\end{proposition}

\begin{theorem}(cf. \cite[Proposition 9.3.2]{LazII}). \label{thm:mult}
Let $D$ be an effective $\Q$-divisor on a smooth projective variety $X$ and let $x$ be a point on $X$.
If  $\mult_x D < 1$,
then $$\mathcal{J}(X,D)_x=\mathcal{O}_{X,x}.$$
\end{theorem}

\begin{theorem}[Nadel's theorem](cf. \cite[Theorem 9.4.17]{LazII}). \label{thm:nadel}
Let $(X,\D)$ be a pair and let $D$ be a $\Q$-divisor on $X$.
If $N$ is a divisor such that $N-(K_X+\D+D)$ is big and nef,
then
$$H^i(X,\mathcal{O}_X(N)\otimes \mathcal{J}((X,\D);D))=0$$
for every $i>0$.

Moreover if $D$ is integral and $\kappa(X,D)\geq 0$, then
$$H^i(X,\mathcal{O}_X(N)\otimes \mathcal{J}((X,\D);\|D\|))=0.$$
\end{theorem}

Given the language of multiplier ideals, we can define KLT pairs quickly and tidily:

\begin{definition}
Let $(X,\Delta)$ be a pair. $(X,\Delta)$ is said to be \emph{Kawamata Log Terminal}, or simply \emph{KLT}, if $\mathcal{J}(X,\Delta)=\mathcal{O}_X$.
More generally the \emph{non-klt locus} of the pair $(X,\D)$ is defined as
$$\mathrm{Nklt}(X,\Delta) := \mathcal{Z}(\mathcal{J}(X,\Delta)).$$
\end{definition}

It is well known that this definition of a KLT pair coincides with \cite[Definition 2.34]{Kollar}.

\subsection{Asymptotic base loci}


We recall the following well-known definitions:

\begin{definition}
Let $X$ be a normal projective variety, let $D$ be an $\R$-divisor on $X$.
\begin{enumerate}
\item $|D|_{\equiv}:=\{E \mid E \text{ effective } \mathbb{R}\text{-divisor}, E \equiv D\}$,
\item $|D|_{\mathbb{R}}:=\{E \mid E \text{ effective } \mathbb{R}\text{-divisor}, E \sim_\mathbb{R} D\}$, where $E \sim_{\mathbb{R}}D$ means that $E-D$ is an
 $\mathbb{R}$-linear combination of principal divisors $(f)$, $f \in \mathbb{C}(X)$,
\item $|D|_{\mathbb{Q}}:=\{E \mid E \text{ effective } \mathbb{R}\text{-divisor}, E \sim_\mathbb{Q} D\}$.
\end{enumerate}
\end{definition}

\begin{definition} (cf.\ \cite[Definition 3.5.1]{BCHM}).
Let $X$ be a normal projective variety, let $D$ be an $\R$-divisor on $X$. The (real) \textit{stable base locus} of  $D$ is
$$\mathbf{B}(D):= \bigcap_{E \in |D|_\mathbb{R}} \supp(E),$$ where, by convention, we put $\mathbf{B}(D)=X$ if $|D|_\mathbb{R}=\emptyset$.

\end{definition}


\begin{definition}(cf. \cite[Definition 1.2]{ELMNP}).
Let $X$ be a normal projective variety, let $D$ be an $\R$-divisor on $X$. The \textit{augmented base locus} of $D$ is
$$\mathbf{B}_+(D):= \bigcap_{\substack{E \  \mathbb{R}\text{-divisor},E \geq 0\\ D-E \text{ ample }}} \supp(E),$$
if $D$ is big; otherwise $\mathbf{B}_+(D):=X$ by convention.
\end{definition}

\begin{definition}(cf. \cite[Definition 1.12]{ELMNP}). \label{bmenodef}
Let $X$ be a normal projective variety, let $D$ be an $\R$-divisor on $X$. The \textit{restricted base locus} of $D$ is
$$\mathbf{B}_{-}(D):= \bigcup_{\substack{A \  \mathbb{R}\text{-divisor}\\ A \text{ ample }}} \mathbf{B}(D+A).$$
\end{definition}

Let now $D$ be a big $\R$-divisor on a normal projective variety $X$.
Given a geometric valuation $v$ on $X$ we define, as in \cite{BBP}, the \emph{numerical vanishing order} of $D$ along $v$ as
$$v_{\num}(D):=\inf\{v(E) \mid E\in |D|_{\equiv} \}.$$
When $D$ is a pseudoeffective $\R$-divisor, as in \cite{BBP}, we set
$$v_\num(D):=\lim_{\varepsilon\rightarrow 0}v_\num(D+\varepsilon A),$$
with $A$ ample.
Note that this definition is just a generalization of \cite[Definition III.2.2]{Nakayama}.
It is easy to see that it does not depend on the choice of the ample divisor $A$. See \cite[Lemma III.1.5(2)]{Nakayama}. 

\begin{definition} (cf. \cite[Definition III.2.6]{Nakayama} and \cite[Definition 1.7]{BBP}).\label{pino}
Let $X$ be a normal projective variety, let $D$ be an $\R$-divisor on $X$ and let us denote by $c_X(v)$ the center on $X$ of a given geometric
valuation $v$ on $X$.
The \emph{non-nef locus} of $D$ is
$$\mathrm{NNef}(D):= \bigcup\{c_X(v) \mid v_\num(D)>0\},$$
if $D$ is pseudoeffective. When $D$ is not pseudoeffective we put $\NNef(D):=X$.
\end{definition}

Note that, as the name itself suggests, an $\R$-divisor $D$ is nef if and only if $\NNef(D)=\emptyset$, i.e.,  if and only if $v_\num(D)=0$ for every
geometric valuation $v$ on $X$ (see \cite[Remark III.2.8]{Nakayama}, \cite[\S 1.3]{BBP}).\\

The following easy lemma about asymptotic base loci and the subsequent one about approximation of $\mathbb{R}$-divisors will allow us to extend results from big $\mathbb{Q}$-divisors to pseudoeffective $\mathbb{R}$-divisors. We denote by $\| \cdot \|$ any fixed norm on $N^1(X)_{\R}$.

\begin{lemma} \label{bmeno}
Let $X$ be a normal projective variety. Let $D$ be an $\mathbb{R}$-divisor on $X$. Let $\{A_m\}_{m\geq 1}$
be any sequence of ample $\mathbb{R}$-divisors such that $\|A_m\| \rightarrow 0$ in $N^1(X)_{\R}$.
Then $$\mathbf{B}_-(D)=\bigcup_{m \geq 1} \mathbf{B}_-(D+A_m)$$ and $$\emph{NNef}(D)=\bigcup_{m \geq 1} \emph{NNef}(D+A_m).$$
\end{lemma}

\begin{proof}
By definition $\mathbf{B}_-(D) = \bigcup_{A \textrm{ ample}} \mathbf{B}(D+A)$ and thus we also have
that $\mathbf{B}_-(D) = \bigcup_{A \textrm{ ample}} \mathbf{B}_-(D+A)$. For any $A$ ample divisor let $m_A$ be sufficiently large
so that $A-A_{m_A}$ is still ample.
Hence $$\mathbf{B}_-(D+A) = \mathbf{B}_-(D+A_{m_A}+A-A_{m_A}) \subseteq \mathbf{B}_-(D+A_{m_A}).$$
Since clearly $\mathrm{NNef}(D)= \bigcup_{A \text{ ample}} \mathrm{NNef}(D+A)$, then the same proof applies to the non-nef locus.
\end{proof}

%

\begin{lemma} \label{successione} 
Let $X$ be a normal projective variety and let $D$ be an $\mathbb{R}$-divisor on $X$. Then there exists a sequence $\{A_m\}_{m\geq 1}$
of ample $\mathbb{R}$-divisors such that $\|A_m\| \rightarrow 0$ in $N^1(X)_\R$ and $D+A_m$ is a $\mathbb{Q}$-divisor for every $m \geq 1$.
\end{lemma}
\begin{proof}
Let $B$ be an ample $\mathbb{R}$-divisor such that $D+B$ is a $\mathbb{Q}$-divisor. By perturbing the coefficients of $B$ we can produce a sequence of $\mathbb{Q}$-divisors $B_m$ such that $B-B_m$ is ample and $\|B-B_m\| \rightarrow 0$. Set $A_m:=B-B_m$ and the statement follows.
%
%
 \end{proof}

\subsection{Asymptotic orders of vanishing}

\begin{definition} \label{asymptotic} (cf. \cite{Nakayama} and \cite[Definition 2.2]{ELMNP}). 
Let $X$ be a normal projective variety and let $D$ be a divisor such that $\kappa(X,D)\geq0$. Let $v$ be a geometric valuation on $X$.
The \textit{asymptotic order of vanishing} of $D$ along $v$ is
$$v(\| D\|):=\lim_{p \rightarrow \infty} \frac{v(|pD|)}{p},$$
where the limit is taken over sufficiently divisible integers $p$'s and $v(|pD|):=v(D')$, where $D'$ is general in $|pD|$. 
\end{definition}

\begin{remark}
Notice that $v(\| \cdot \|)$ is homogeneous and convex: for every $D,D'$ of non-negative Kodaira dimension, for every $k \in \mathbb{N}$,
\begin{align}
 &v(\|kD\|) =kv(\|D\|), \tag{homogeneity}\\
&v(\|D+D'\|) \leq v(\|D\|)+v(\|D'\|) \tag{convexity}.
\end{align} See \cite[Remark 2.3, Proposition 2.4]{ELMNP}.

In particular, by homogeneity, the above definition can be generalized to $\mathbb{Q}$-divisors and it can be easily seen that the limit is also the inf,
so that for every $\Q$-divisor $D$ we have that
$$v(\| D\|)=\inf\{v(E) \mid E \in |D|_\Q \}.$$
\end{remark}

\begin{remark} \label{casobig}
If $D$ is a big $\mathbb{Q}$-divisor, then $$v(\|D\|)=v_{\mathrm{num}}(D)=\inf\{v(E) \mid E \in |D|_\mathbb{R}\}$$ (see \cite[Lemma III.1.4]{Nakayama}, \cite[Lemma 3.3]{ELMNP}).
More generally, when $D$ is \textit{abundant}
we have that $v(\|D\|)=v_{\mathrm{num}}(D)$
by \cite[Proposition 6.4]{Lehmann}.

See \cite[Definition V.2.23]{Nakayama} for the definition of abundant divisor, see \cite{Lehmann2} and \cite{Lehmann} for some equivalent definitions and remarks.

In general, if $D$ is an effective $\mathbb{Q}$-divisor, then by definition $v(\|D\|) \geq v_{\text{num}}(D)$
but equality does not always hold.
Take for example a nef irreducible curve $D$ on a smooth surface as in \cite[Example 1]{R} and set $v:=\ord_D$.
We have that $v_{\textrm{num}}(D)=0$ by the nefness of $D$, while $v(\|D\|)= 1$.
\end{remark}

Recall that, given a nef divisor $D$ on a normal projective variety $X$, we can define its \textit{numerical dimension} as
$\nu(X,D):=\max\{k \in \mathbb{N}:D^k \not \equiv 0\}.$ If $D$ is nef, we have that it is \textit{abundant}
if and only if the Kodaira dimension of $D$ equals its numerical dimension (see \cite[Proposition V.2.22(5)]{Nakayama}).

In this paper we will only deal with abundance of nef divisors.
The following lemma is a translation of \cite[Theorem 1]{R} in terms of discrete valuations:

\begin{lemma}\label{Russo_valuations}
Let $D$ be a divisor on a projective normal variety $X$.
Then $D$ is nef and abundant if and only if $v(\|D\|)=0$ for every geometric valuation $v$ on $\mathbb{C}(X)$.
\end{lemma}

\begin{proof}
If $D$ is nef and abundant, then by \cite[Lemma 1]{MR} (see also \cite[Proposition 2.1]{Kawamata})
there exist a birational morphism $f:Z\rightarrow X$, where $Z$ is a smooth projective variety, an integer $k_0>0$
and an effective divisor $N$ on $Z$ such that $B_m:=mk_0f^*(D)-N$ is semiample for every $m\in \N$.

Now, given any geometric valuation $v$ on $\C(X)$, $v(\|D\|)=v(\|f^*D\|)$. 
For every $m$, by the homogeneity and convexity of the asymptotic order of vanishing,
$$v(\|f^*D\|)=\frac{1}{mk_0} v(\|f^*(mk_0 D)\|) \leq \frac{1}{mk_0} (v(\|B_m\|)+v(\|N\|)),$$
thus $v(\|f^*D\|)=0$ because $v(\|B_m\|)=0$ by the semiampleness of $B_m$ and $v(\| N \|)$ does not depend on $m$.

Now, suppose $v(\|D\|)=0$ for every geometric valuation $v$ on $\mathbb{C}(X)$.
If $\mu:X'\rightarrow X$ is a desingularization of $X$, then $v(\|\mu^*D\|)=v(\|D\|)=0$ for every $v$,
which implies that $\mu^*D$ is almost base point free (cf. \cite[Definition 1]{R}).
By \cite[Theorem 1]{R} this proves that $\mu^*D$ is nef and abundant, which in turn implies that $D$ is nef and abundant.
\end{proof}

The previous lemma justifies the following definition:

\begin{definition}\label{gigetto}
Let $X$ be a normal projective variety and let $D$ be a divisor such that $\kappa(X,D)\geq0$.
The \emph{non nef-abundant locus} of $D$ is $$\mathrm{NNA}(D):=\bigcup_{v \in V} \{c_X(v) \mid v(\|D\|)>0\},$$
where $V$ is the set of all geometric valuations on $\mathbb{C}(X)$ and, for any $v \in V$, $c_X(v)$ is the center of $v$
on $X$.
\end{definition}

\begin{remark} \label{nnaennef}
Trivially, by Lemma \ref{Russo_valuations}, $\mathrm{NNA}(D)=\emptyset$ if and only if $D$ is nef and abundant.
When $D$ is a big divisor, $\mathrm{NNef}(D)=\mathrm{NNA}(D)$
by Remark \ref{casobig}. By the same remark we see that if $D$ is effective, then in general $\mathrm{NNA}(D) \supseteq \mathrm{NNef}(D)$
but if $D$ is not big, then equality does not always hold.
\end{remark}

\subsection{Birational maps} Since we want to compare asymptotic base loci and zeroes of multiplier ideals on singular varieties, the first thing to do is to reduce ourselves to a convenient desingularization. The following three lemmas will be used later on for such a purpose:
\begin{lemma} \label{immagine nna}
Let $X$ be a normal projective variety and let $D$ be a divisor such that $\kappa(X,D)\geq0$.
Let $f:X' \rightarrow X$ be any birational morphism from a normal variety $X'$. Then $f(\mathrm{NNA}(f^*D))=\mathrm{NNA}(D)$.
\end{lemma}

\begin{proof}
The lemma follows from the easy fact that, for any geometric valuation $v$ on $\mathbb{C}(X)$, $v(\|f^*D\|)=v(\|D\|)$.
\end{proof}
\begin{lemma}\label{immagine ideale}
Let $X,Y$ be normal varieties and let $f: Y \rightarrow X$ be a birational morphism. Let $\mathcal{J}$ be a coherent sheaf of ideals on $Y$.
Then $\mathcal{Z}(f_*\mathcal{J}) \subseteq f(\mathcal{Z}(\mathcal{J}))$.
\end{lemma}

\begin{proof}
Let $W:=Y \setminus \mathcal{Z}(\mathcal{J})$ and $V:=X \setminus f(\mathcal{Z}(\mathcal{J}))$.
For every $V' \subseteq V$ open subset of $X$ we have that $f^{-1}(V') \subseteq W$,
hence $$(f_*\mathcal{J})(V')=\mathcal{J}(f^{-1}(V'))=\mathcal{O}_Y(f^{-1}(V'))=f_*\mathcal{O}_Y(V') \cong \mathcal{O}_X(V')$$
by Zariski's main theorem. 
Therefore if $x \not \in f(\mathcal{Z}(\mathcal{J}))$, then $(f_*\mathcal{J})_x = \mathcal{O}_{X,x}$, i.e., $x \not \in \mathcal{Z}(f_*\mathcal{J})$ and we are done.
\end{proof}

\begin{lemma}[Birational transformation rule for asymptotic multiplier ideals] \label{transformation rule} Let $(X,\Delta_X)$, $(Y,\Delta_Y)$ be pairs and let $f:Y \rightarrow X$ be a birational morphism.
Assume that $$K_Y+\Delta_Y \equiv f^*(K_X+\Delta_X) \textrm{ and } f_*\Delta_Y=\Delta_X.$$
If $D$ a divisor on $X$ such that $\kappa(X,D)\geq 0$, then $$\mathcal{J}((X,\Delta_X);\|D\|)=f_*\left(\mathcal{J}((Y,\Delta_Y);\|f^*D\|)\right).$$
\end{lemma}
\begin{proof}
We can choose a sufficiently large and divisible $k\in \N$ and a  general $D_k \in |kD|$ such that
\begin{align*}
&\mathcal{J}((X,\Delta_X);\|D\|)  =  \mathcal{J}((X,\Delta_X);\frac{1}{k}D_k), \\
&\mathcal{J}((Y,\Delta_Y);\|f^*D\|)  = \mathcal{J}((Y,\Delta_Y);\frac{1}{k}f^*(D_k)).
\end{align*}
Then the statement follows by the usual birational transformation rule for multiplier ideals (\cite[Proposition 9.3.62]{LazII}).
%
\end{proof}

\section{Some special cases}

In this section we investigate the relationship between $\mathbf{B}_-(D)$ and $\NNef(D)$ just exploiting the fact that, by Theorem \ref{thm:casoliscio}, we already know that they are equal on smooth varieties. After a few lemmas about the behavior of the restricted base locus under birational maps, we prove that $\mathbf{B}_-(D)$ and $\NNef(D)$ agree on the smooth locus of $X$. Some considerations will then allow us to conclude that $\mathbf{B}_-(D)=\NNef(D)$ on any normal surface.





\begin{lemma} \label{b+}
Let $X$ and $Y$ be normal projective varieties and let $f: Y \rightarrow X$ be a birational morphism.
If $A$ is an ample $\mathbb{R}$-divisor on $X$, then $\mathbf{B}_+(f^*A)=\mathrm{Exc}(f)$.
\end{lemma}
\begin{proof}
See \cite[Proposition 1.5]{BBP}.
\end{proof}

\begin{lemma} \label{b-}
Let $X$ and $Y$ be normal projective varieties and let $f: Y \rightarrow X$ be a birational morphism.
If $D$ is an $\mathbb{R}$-divisor on $X$, then  $ f^{-1}(\mathbf{B}_-(D)) \setminus \exc(f) \subseteq  \mathbf{B}_-(f^*D)$.
\end{lemma}

\begin{proof}
Let $y \in f^{-1}(\mathbf{B}_-(D)) \setminus \mathrm{Exc}(f)$. The fact that $f(y) \in \mathbf{B}_-(D)$ implies, by definition, that there exists $A_X$, an ample $\mathbb{R}$-divisor on $X$, such that $f(y) \in \mathbf{B}(D+A_X)$.
Therefore $y \in \mathbf{B}(f^*D+f^*A_X)$.
By Lemma \ref{b+} there exist $A_Y$, an ample $\mathbb{R}$-divisor on $Y$, and $E_Y$,
an effective $\mathbb{R}$-divisor on $Y$, such that $f^*A_X=A_Y+E_Y$ and $y \not \in \supp(E_Y)$. Since $\mathbf{B}(f^*D+A_Y+E_Y) \subseteq \mathbf{B}(f^*D+A_Y) \cup \mathbf{B}(E_Y)$ and $ y \not \in \supp(E_Y)$, we have that $y \in \mathbf{B}(f^*D+A_Y) \subseteq \mathbf{B}_-(f^*D)$.
\end{proof}

We can now compare $\mathbf{B}_-(D)$ and $\mathrm{NNef}(D)$ on the smooth locus of $X$.
To this purpose define $X_{\mathrm{sm}}$ to be the \textit{smooth locus} of $X$, i.e., $X_{\mathrm{sm}}:=X \setminus \mathrm{Sing}(X)$.
The following holds:

\begin{proposition} \label{smooth}
Let $X$ be a normal projective variety and let $D$ be an $\mathbb{R}$-divisor on $X$.
We have that $\mathbf{B}_-(D) \cap X_{\mathrm{sm}} = \mathrm{NNef}(D) \cap X_{\mathrm{sm}}$.
\end{proposition}

\begin{proof}
In general $\mathrm{NNef}(D) \subseteq \mathbf{B}_-(D)$ (see \cite[Lemma 1.8]{BBP}), thus it is enough to show that
$\mathbf{B}_-(D) \cap X_{\mathrm{sm}} \subseteq \mathrm{NNef}(D)$.

Let $f:Y \rightarrow X$ be a resolution of the singularities of $X$ constructed as a series of blow-ups along smooth centers
contained in $\mathrm{Sing}(X)$ (this is possible by Hironaka's theorem - cf. \cite[Theorem 4.1.3]{LazI}).
By Lemma \ref{b-}, $f(\mathbf{B}_-(f^*D)) \supseteq \mathbf{B}_-(D) \cap X_{\mathrm{sm}}$.
Since $Y$ is smooth, then $\mathbf{B}_-(f^*D)=\mathrm{NNef}(f^*D)$ by Theorem \ref{thm:casoliscio},
therefore $$\mathrm{NNef}(D) = f(\mathrm{NNef}(f^*D))=f(\mathbf{B}_-(f^*D)) \supseteq \mathbf{B}_-(D) \cap X_{\mathrm{sm}},$$
where the first equality is a straightforward consequence of \cite[Lemma 1.6]{BBP} (or it easily follows from Lemma \ref{immagine nna}).
\end{proof}

Note that in the following section we will give a generalization of this result (see Corollary \ref{corb-ennef} and Remark \ref{Gongyo}).

Recall that, for every normal variety $X$ and $\mathbb{R}$-divisor $D$ on $X$, we have that
both $\mathbf{B}_-(D)$ and $\mathrm{NNef}(D)$ are at most a countable union of Zariski closed subsets of $X$ (see, for example, \cite{BBP}).
Therefore the following holds:

\begin{corollary}\label{divisorial}
Let $X$ be a normal projective variety and let $D$ be an $\mathbb{R}$-divisor on $X$. Then every divisorial component of $\mathbf{B}_-(D)$ is contained in $\mathrm{NNef}(D)$.
\end{corollary}
\begin{proof}
Let $E$ be a prime divisor on $X$ such that
$E \subseteq \mathbf{B}_-(D)$. Write $\mathrm{NNef}(D)=\bigcup_{i \in \mathbb{N}} V_i$, $V_i$ proper subvariety of $X$. By Proposition \ref{smooth},
$$E \cap X_{\mathrm{sm}} = \bigcup_{i \in \mathbb{N}} (E \cap X_{\mathrm{sm}} \cap V_i).$$ Since $\codim(\mathrm{Sing}(X)) \geq 2$, then $E \cap X_{\mathrm{sm}} \not = \emptyset$. Hence, by Baire's category theorem applied to $E \cap X_{\mathrm{sm}}$ (with the Euclidean topology), 
there must exist $j \in \mathbb{N}$ such that $E \cap X_{\mathrm{sm}} \subseteq V_j$. By taking closures we get that $E=V_j$.
\end{proof}


\begin{corollary}\label{superfici}
Let $X$ be a normal projective surface and let $D$ be an $\mathbb{R}$-divisor on $X$. Then $\mathbf{B}_-(D)=\mathrm{NNef}(D)$.
\end{corollary}
\begin{proof}
By \cite[Proposition 1.1]{ELMNPrestricted} it follows that $\mathbf{B}_-(D)$ has no isolated points.
Since $X$ is a surface, this is equivalent to saying that $\mathbf{B}_-(D)$ has only divisorial components,
so that we can conclude by Corollary \ref{divisorial}.
\end{proof}

\section{Main results}


In this section we will prove Theorem \ref{riassuntivo}; this is done in Theorem \ref{b-ennef}.
The idea is to prove that, given an effective KLT pair $(X,\Delta)$ and an effective integral divisor $D$, we have that

$$\B_-(D) \subseteq \bigcup_p \mathcal{Z}(\mathcal{J}((X,\Delta);\|pD\|)) \subseteq \mathrm{NNA}(D). $$
The former inclusion is a consequence of Nadel's vanishing theorem and
the proof is just an easy generalization to singular varieties of some arguments in \cite[Proposition 2.8]{ELMNP}.
This is the content of Lemma \ref{nadel}.


To prove the latter inclusion we notice that, by considering a suitable log-resolution, we can reduce
to the smooth case and get rid of the boundary $\D$ at the same time,
so that the result follows as an application of Theorem \ref{thm:mult} (see Proposition \ref{spariscedelta} and Theorem \ref{principale}).

The rest of the section is devoted to slight generalizations and corollaries.

\begin{lemma} \label{nadel}
Let $(X,\Delta)$ be an effective pair. Let $D$ be an integral divisor such that $\kappa(X,D) \geq 0$.
Then $\mathbf{B}_-(D) \subseteq \bigcup_p \mathcal{Z}(\mathcal{J}((X,\Delta);\|pD\|))$.
\end{lemma}
\begin{proof}
We will follow \cite[proof of Proposition 2.8]{ELMNP}, taking into account the fact that $X$ may be singular.
Take $x\in X$ such that $x \not \in \mathcal{Z}(\mathcal{J}((X,\Delta);\|pD\|))$  for every $p \geq 1$. We will show that $x \not \in \mathbf{B}_-(D)$.

Let $A$ be a fixed very ample divisor such that $A-(K_X+\Delta)$ is ample.
Set $n=\dim X$. By Nadel's vanishing theorem in the singular setting (see Theorem \ref{thm:nadel}),
we have that, for every $i \geq 1$,
$$H^i \left( X, \mathcal{O}_X \left( (n+1)A+pD\right) \otimes \mathcal{O}_X \left( -iA \right)\otimes \mathcal{J} ((X,\Delta);\|pD\|)\right) =0.$$
Thus, by Castelnuovo--Mumford's criterion (see \cite[Theorem 1.8.5]{LazI} or \cite[p.\ 194]{LazII}),
we have that $\mathcal{O}_X \left( (n+1)A+pD\right) \otimes \mathcal{J} ((X,\Delta);\|pD\|)$ is globally generated.
But, by hypothesis, $\mathcal{J}((X,\Delta);\|pD\|)_x =\mathcal{O}_{X,x}$.
Therefore we get that $x \not \in \text{Bs}|(n+1)A+pD|$ for every $p \geq 1$, i.e., $x \not \in \mathbf{B}_-(D)$.
\end{proof}

\begin{proposition} \label{spariscedelta}
Let $(X,\D)$ be an effective KLT pair such that $X$ is smooth and $\mult_x(\D)<1$ for every $x \in X$. If $D$ is a divisor on $X$ such that $\kappa(X,D)\geq 0$, then $\mathcal{Z}(\mathcal{J}((X,\D);\|D\|))\subseteq \mathrm{NNA}(D)$.
\end{proposition}

\begin{proof}
For every $k$ sufficiently large and divisible, and $D_k \in |kD|$ general, \linebreak $\mathcal{J}((X,\D);\|D\|)=\mathcal{J}((X,\D);\frac{1}{k}D_k)$. Now, if $x \in \mathcal{Z}(\mathcal{J}((X,\D);\frac{1}{k}D_k))$, then, by Theorem \ref{thm:mult}, $\mult_x(\D+ \frac{1}{k} D_k) \geq 1$. Since $\mult_x$ is additive and, by the hypothesis on $\D$, $\mult_x(\D) =1-c_x$ for a certain $c_x >0$, then $\mult_x( \frac{1}{k} D_k) \geq c_x$.
Thus $$\mult_x(\|D\|)=\lim_{k\to \infty} \mult_x( \frac{1}{k} D_k)\geq c_x>0,$$ so that $x \in \mathrm{NNA}(D)$.
\end{proof}

\begin{theorem}\label{principale}
Let $(X,\D)$ be an effective KLT pair.
If $D$ is a divisor on $X$ such that $\kappa(X,D)\geq 0$, then
$$\bigcup_{p\in \N} \mathcal{Z}(\mathcal{J}((X,\D);\|pD\|))\subseteq \mathrm{NNA}(D).$$

\end{theorem}

\begin{proof}
We have to show that, for every positive integer $p$,
we have that
\begin{equation} \label{inclusione}
\mathcal{Z}(\mathcal{J}((X,\D);\|pD\|))\subseteq \mathrm{NNA}(D).
\end{equation}

Since by definition $\mathrm{NNA}(pD)=\mathrm{NNA}(D)$,
without loss of generality we can assume that $p=1$.

As $(X,\D)$ is a KLT pair, by \cite[Proposition 2.36]{Kollar}, there exists $f:Y\to X$, a log-resolution of $(X,\D)$, such that
we can write $$K_Y+\D'_Y=f^*(K_X+\D)+E,$$
where $\D'_Y$ and $E$ are effective $\Q$-divisors without common components and the components of $\D'_Y$ are disjoint.
In particular, $Y$ is smooth and $\mult_y \D'_Y<1$ for every $y \in Y$.

 Thus, by Proposition \ref{spariscedelta}, we have that
 \begin{equation} \label{inclusione_sopra}
\mathcal{Z}(\mathcal{J}((Y,\D'_Y);\|f^*D\|))\subseteq \mathrm{NNA}(f^*D).
\end{equation}
Now, since $\D'_Y \geq \D'_Y-E$, then $\mathcal{J}((Y,\Delta'_Y-E);\|f^*D\|)\supseteq \mathcal{J}((Y,\Delta'_Y);\|f^*D\|)$.

Hence, by Lemma \ref{transformation rule}, we have that $$\mathcal{J}((X,\Delta);\|D\|)=f_*\left(\mathcal{J}((Y,\Delta'_Y-E);\|f^*D\|)\right)\supseteq f_* \left( \mathcal{J}((Y,\Delta'_Y);\|f^*D\|) \right),$$ so that
$$\mathcal{Z}\left(\mathcal{J}((X,\Delta);\|D\|) \right) \subseteq \mathcal{Z}\left( f_* \left( \mathcal{J}((Y,\Delta'_Y);\|f^*D\|) \right)\right).$$

But, by Lemma \ref{immagine ideale},  $\mathcal{Z} \left( f_* \left( \mathcal{J}((Y,\Delta'_Y);\|f^*D\|) \right) \right) \subseteq
f \left( \mathcal{Z} \left( \mathcal{J}((Y,\Delta'_Y);\|f^*D\|) \right) \right)$ and, by Lemma \ref{immagine nna}, $f(\mathrm{NNA}(f^*D))=
\mathrm{NNA}(D)$, so that (\ref{inclusione}) follows by (\ref{inclusione_sopra}) and we are done.
\end{proof}

If the effective pair $(X,\Delta)$ is not KLT, the same statement does not hold in general, because
the zeroes of the asymptotic multiplier ideals depend also on the singularities of the pair.
Anyway, we can still recover the previous result outside the non-klt locus.
More precisely the following holds:

\begin{corollary}\label{corprincipale}
Let $(X,\D)$ be an effective pair.
If $D$ is a divisor on $X$ such that $\kappa(X,D)\geq 0$, then
$$\bigcup_{p\in \N} \mathcal{Z}(\mathcal{J}((X,\D);\|pD\|)) \setminus \mathrm{Nklt}(X,\Delta) \subseteq \mathrm{NNA}(D).$$
\end{corollary}

\begin{proof}
As in the proof of Theorem \ref{principale} we can assume $p=1$, i.e., it is enough to show that
$\mathcal{Z}(\mathcal{J}((X, \Delta);\|D\|)) \setminus \mathrm{Nklt}(X,\D) \subseteq \mathrm{NNA}(D)$.
Let $f: Y \rightarrow X$ be a log-resolution of $(X,\Delta)$ and let $\Delta_Y$ be such that $K_Y+\Delta_Y \equiv f^*(K_X+\Delta)$.
Set $\Delta_Y= \sum a(E) E$, where the sum is taken on all prime divisors on $Y$, and define
$$\Delta_Y^{\geq 1} := \sum_{a(E) \geq 1}a(E)E,\;
\Delta_Y^+:=\sum_{0 \leq a(E) < 1}a(E)E, \; \Delta_Y^-:= \sum_{a(E) < 0}-a(E)E,$$
$$\Delta_Y':=\Delta_Y^{\geq 1} + \Delta_Y^+.$$
Notice that $\Delta_Y=\Delta_Y^{\geq 1}+\Delta_Y^+ - \Delta_Y^-$ and $f(\mathrm{Nklt}(Y,\Delta_Y))=\mathrm{Nklt}(X, \Delta)$.

As in the proof of Theorem \ref{principale}, by the birational transformation rule, Lemma \ref{immagine ideale} and Lemma \ref{immagine nna},
it is then enough to prove that $$\mathcal{Z}(\mathcal{J}((Y,\Delta_Y');\|f^*D\|))\setminus \mathrm{Nklt}(Y,\Delta_Y) \subseteq \mathrm{NNA}(f^*D).$$
At this point, notice that, for any $y \not \in \mathrm{Nklt}(Y,\Delta_Y)=\supp(\Delta_Y^{\geq 1})$, we have that
$\mathcal{J}((Y,\Delta_Y');\|f^*D\|)_y \cong \mathcal{J}((Y,\Delta_Y^+);\|f^*D\|)_y $. 
Therefore we are just left to prove that
$$\mathcal{Z}(\mathcal{J}((Y,\Delta_Y^+);\|f^*D\|)) \subseteq \mathrm{NNA}(f^*D),$$
but, since $(Y,\Delta_Y^+)$ is an effective KLT pair, this is just an instance of Theorem \ref{principale}, and we are done.
\end{proof}

\begin{theorem}\label{b-ennef}
Let $(X,\Delta)$ be an effective pair and let $D$ be an $\mathbb{R}$-divisor on $X$.
Then $\mathrm{NNef}(D) \setminus \mathrm{Nklt}(X,\Delta)=\mathbf{B}_-(D) \setminus \mathrm{Nklt}(X,\Delta)$.
In particular, if $(X,\Delta)$ is an effective KLT pair, then $\mathrm{NNef}(D)=\mathbf{B}_-(D)$.
\end{theorem}

\begin{proof}
If $D$ is not pseudoeffective there is nothing to prove. Thus let us assume that $D$ is pseudoeffective.
Let $\{A_m\}$ be a sequence of ample $\mathbb{R}$-divisors as in Lemma \ref{successione}.
By Lemma \ref{bmeno}, it is then clear that we can furthermore assume that $D$ is a big $\mathbb{Q}$-divisor.
Since for every $c>0$, $\mathbf{B}_-(cD)=\mathbf{B}_-(D)$, and analogously for $\text{NNef}$, we can also assume that $D$ is integral.
By \cite[Lemma 1.8]{BBP} and Lemma \ref{nadel} we have that
$$\mathrm{NNef}(D) \subseteq \mathbf{B}_-(D) \subseteq \bigcup_p \mathcal{Z}(\mathcal{J}((X,\Delta);\|pD\|)).$$
Since $D$ is big, by Remark \ref{nnaennef}, $\mathrm{NNA}(D) = \mathrm{NNef}(D)$, whence, by Corollary \ref{corprincipale},
we have that $$ \bigcup_p \mathcal{Z}(\mathcal{J}((X,\Delta);\|pD\|)) \setminus \mathrm{Nklt}(X,\Delta) \subseteq \mathrm{NNef}(D)$$ and we are done.
\end{proof}

\begin{remark}
When $X$ is smooth Theorem \ref{b-ennef} has been proved by Nakayama in \cite{Nakayama} (see also 
\cite{ELMNP}).

If $(X, \Delta)$ is an effective KLT pair, then Theorem \ref{b-ennef} has been proved to hold for the divisor $K_X+\Delta$ by
Boucksom, Broustet, Pacienza in \cite[Proposition 1.10]{BBP} using \cite{BCHM}.

In general, it has been conjectured that the equality $\NNef(D)=\B_-(D)$ holds for every pseudoeffective $\R$-divisor on any normal projective variety: see, for example, \cite[Conjecture 1.9]{BBP}.
\end{remark}


Since $\mathrm{NNef}(D)$ and $\mathbf{B}_-(D)$ do not depend on the chosen boundary divisor $\Delta$,
if we set $X_{\mathrm{Nklt}}:=\bigcap_{\Delta \in \mathcal{F}} \mathrm{Nklt}(X,\Delta)$, where
$$\mathcal{F}:=\{E  \ \mathbb{Q}\text{-Weil divisor} \mid (X,E) {\text{ is an effective pair}}\},$$
then the following holds:

\begin{corollary}\label{corb-ennef}
Let $X$ be a normal projective variety and let $D$ be an $\mathbb{R}$-divisor on $X$.
Then $\mathrm{NNef}(D) \setminus X_{\mathrm{Nklt}} = \mathbf{B}_-(D) \setminus X_{\mathrm{Nklt}}$.
\end{corollary}

\begin{remark}\label{Gongyo}
Note that Proposition \ref{smooth} follows also directly by Corollary \ref{corb-ennef} because, for every normal variety $X$
and any smooth closed point $x \in X$, there exists an effective pair $(X,\D_x)$ such that $x\not\in \mathrm{Nklt}(X,\D_x)$.

More precisely, as Y.\ Gongyo kindly pointed out to us,
we can find an effective $\mathbb{Q}$-Weil divisor $\Delta_x$ such that $(X,\Delta_x)$ is a pair and $x \not \in \supp(\Delta_x)$.

In fact, if $H$ is an ample Cartier divisor, then the coherent sheaf $\mathcal{O}_X(-K_X+mH)$ is globally generated for $m\in \N$ large enough.
On the other hand $\mathcal{O}_X(-K_X+mH)_x\simeq \mathcal{O}_{X,x}$, because $x\in X$ is a smooth point,
so that there exists a section $s_x\in H^0(X,\mathcal{O}_X(-K_X+mH))$ not vanishing on $x$
and we can just take $\D_x=\{s_x=0\}$.
Notice that $K_X+\D_x\sim mH$ is a Cartier divisor, i.e., $(X, \Delta_x)$ is a pair. 


\end{remark}

Note that, as $\B_-(D)$ does not contain isolated points (see \cite[Proposition 1.1]{ELMNPrestricted}), by Corollary \ref{corb-ennef} we deduce the
following:

\begin{corollary}\label{dim0}
Let $X$ be a normal projective variety such that $X_\mathrm{Nklt}$ has dimension 0.
Then for every $\R$-divisor $D$ on $X$ we have that $\mathrm{NNef}(D)=\mathbf{B}_-(D)$.

\end{corollary}

\section{Nef and abundant divisors}
\subsection{Characterization of nef-abundant divisors}

In \cite[Theorem 2]{R} F.\ Russo states a characterization of nef and abundant divisors on a smooth projective variety $X$
by means of asymptotic multiplier ideals. Given Theorem \ref{principale} and its counterpart below (Theorem \ref{secondario}),
we can extend the same characterization to KLT pairs.
\begin{theorem} \label{secondario}
Let $(X,\D)$ be an effective pair.
If $D$ is a divisor on $X$ such that $\kappa(X,D)\geq 0$, then
$$\mathrm{NNA}(D) \subseteq \bigcup_{p\in \N} \mathcal{Z}(\mathcal{J}((X,\D);\|pD\|)). $$
\end{theorem}

\begin{proof}

By taking a multiple of $D$, without loss of generality, we can assume that $|kD|\not =\emptyset$ for every $k\in \N$.

Assume that $x\in \mathrm{NNA}(D)$. Then  there exists a geometric valuation $v$ such that $v(\|D\|)=\nu>0$ and $x \in c_X(v)$.
Let $\mu:X'\to X$ be a resolution of singularities and let $F\subseteq X'$ be the prime divisor corresponding to $v$.
As
$$\nu=v(\|D\|)=\inf_{k \in \N} \left\{\frac{v(|kD|)}{k}\right\},$$
we have that $v(|kD|)\geq k\nu$ for every $k\in \N$.
Now, let us write
$$K_{X'}=\mu^*(K_X+\D)+\sum a(E) \cdot E,$$
where the sum is taken over all prime divisors $E\subseteq X'$, and let us define
$$
p_0:=\left\{\begin{array}{ll} \left\lceil{ \frac{1+a(F)}{\nu} }\right\rceil & \mbox{if } a(F)>-1
\\ 1 & \mbox{if } a(F)\leq -1. \end{array}\right.
$$
Take $h \in \mathbb{N}$ such that $\mathcal{J}((X,\D);\|p_0D\|)=
\mathcal{J}\left((X,\D);\frac{1}{h}D'\right)$, for a general divisor $D'\in |hp_0D|$.
If we write $\mu^*(D')=\sum b(E)\cdot E$,
then
$$K_{X'}-\mu^*\left(K_X+\D+\frac{1}{h}D'\right)=\sum \left(a(E)-\frac{1}{h} b(E)\right) \cdot E.$$

But now $b(F)= v(|hp_0D|)\geq h p_0 \nu$, so that
$a(F)-\frac{1}{h}b(F)\leq a(F)-p_0\nu\leq-1$.
This implies that
$$\mu(F)\subseteq \mathrm{Nklt}\left(X,\D+ \frac{1}{h} D'\right)=\mathcal{Z}\left(\mathcal{J}(X,\D+\frac{1}{h} D')\right)=
\mathcal{Z}\left(\mathcal{J}((X,\D);\|p_0D\|)\right).$$
Since $x\in c_X(v)=\mu(F)$, we are done.
\end{proof}

\begin{corollary} \label{cor:merging}
Let $(X,\D)$ be an effective KLT pair.
If $D$ is a divisor on $X$ such that $\kappa(X,D)\geq 0$, then
$$\mathrm{NNA}(D)=\bigcup_{p\in \N} \mathcal{Z}(\mathcal{J}((X,\D);\|pD\|)). $$
\end{corollary}

\begin{proof}
Just merge together Theorem \ref{principale} and Theorem \ref{secondario}.
\end{proof}

As a particular case, we get the following:

\begin{corollary} \label{charnefabundant}
Let $(X,\D)$ be an effective KLT pair.
Suppose $D$ is a divisor on $X$ such that $\kappa(X,D) \geq 0$.
Then $D$ is nef and abundant if and only if  $\mathcal{J}((X,\Delta);\|pD\|)=\mathcal{O}_X$ for every $p \in \mathbb{N}$.

\end{corollary}




\subsection{Applications}
Sometimes, when one needs to cope with a line bundle $L$ up to $\mathbb{Q}$-linear equivalence on a KLT pair $(X,\Delta)$,
it could be useful to reduce oneself to just forgetting the divisor and to studying a slightly different KLT pair $(X, \Delta')$, in which the boundary $\Delta' = \Delta + D$ ``has absorbed'' the divisor, i.e., $D \sim_{\mathbb{Q}} L$. It is well known that this can be easily done when $L$ is nef and big. By the lemma below, given Corollary \ref{charnefabundant}, the same is true even for nef and abundant divisors:

\begin{lemma} \label{idealebanale}
Let $(X, \Delta)$ be an effective KLT pair and let $L$ be a line bundle on $X$ such that $\kappa(X,L) \geq 0$.
Then $\mathcal{J}((X,\Delta);\|L\|)= \mathcal{O}_X$ if and only if
there exists an effective $\mathbb{Q}$-divisor $D$ such that $D \sim_{\mathbb{Q}} L$ and $(X, \Delta+D)$ is a KLT pair.
\end{lemma}
\begin{proof} 
In general $\mathcal{J}((X,\Delta); \| L\|) \supseteq \mathcal{J}((X,\Delta);D)$ for every effective $\Q$-divisor $D$ such that $D \sim_\mathbb{Q} L$, and equality holds for certain such $D$. Moreover, by definition, $\mathcal{J}((X,\Delta);D)=\mathcal{O}_X$ if and only if $(X,\Delta+D)$ is KLT.
\end{proof}

To illustrate the general principle touched on before, we give a slight generalization
of a theorem by F. Campana, V. Koziarz, M. P\u aun, but only in the case of KLT pairs (see \cite[Corollary 1]{CKP}).
\begin{theorem}[Campana, Koziarz, P\u aun] \label{CKP}
Let $(X,\Delta)$ be an effective KLT pair of dimension $n$ and let $\rho$ be a $\mathbb{Q}$-divisor on $X$ such that
$\rho \equiv 0$. If $L$ is a nef and abundant line bundle on $X$,
then $\kappa(X,K_X+\Delta+L)\geq \kappa(X,K_X+\Delta+L+\rho)$.

In particular, the same holds if we assume that $L$ is nef and $\kappa(X,L) \geq n-1$.
\end{theorem}
\begin{proof}
Since $L$ is nef and abundant, then, by Corollary \ref{charnefabundant} and Lemma \ref{idealebanale}, there exists $\Delta'$
such that $(X,\Delta')$ is an effective KLT pair and $K_X+ \Delta' \sim_{\mathbb{Q}} K_X+\Delta + L$.
Hence $$\kappa(X,K_X+\Delta+L)=\kappa(X,K_X+\Delta')\geq \kappa(X,K_X+\Delta'+\rho),$$ by \cite[Corollary 1]{CKP}.

The last sentence follows because the hypothesis on the Kodaira dimension actually implies that $L$ is nef and abundant: in fact in general $\nu(X,L)\geq \kappa(X,L)$ (see \cite[Proposition 2.2]{Kawamata}), but if $\nu(X,L)=n$, then $L^n > 0$, i.e., $L$ is big. 
\end{proof}
Notice that the hypothesis $\kappa(X,L) \geq n-1$ is necessary. In fact, for every $n \geq 2$, we can find examples of smooth varieties of dimension $n$ and line bundles of Kodaira dimension $n-2$ for which Theorem \ref{CKP}, with $\D=0$, does not hold:
\begin{example} 
We will first of all construct an example for $n=2$.

Let $C$ be a smooth elliptic curve and let $\eta \in \mathrm{Pic}^0(C)$ be a non-torsion divisor on $C$.
Let $\mathcal{E}:=\mathcal{O}_C \oplus  \mathcal{O}_C(-\eta)$.
Take $X:=\mathbb{P}(\mathcal{E})$ and let $\pi: X \rightarrow C$ be the related projection.
As in \cite[Notation V.2.8.1]{Hartshorne} let $C_0$ be a section $\sigma_0: C \rightarrow X$. Set $\rho:=-\pi^*(\eta)$ and $L:= - (K_X+\rho)$.


By \cite[Lemma V.2.10]{Hartshorne}, $K_X \sim -2C_0+\rho$ and $L \sim 2C_0-2\rho$,
so that $L$ is nef because $C_0^2=0$ and $\rho\equiv 0$. 

Hence, for any $m \geq 1$, $H^0(X, mL)=H^0(X, 2mC_0-2m\rho)$. By projection formula
\begin{eqnarray*}
H^0(X, mL)&=&H^0(C, S^{2m}(\mathcal{E}) \otimes 2m\eta)\\
&=&H^0(C, (\mathcal{O}_C \oplus \mathcal{O}_C(-\eta) \oplus \cdots \oplus \mathcal{O}_C(-2m\eta))\otimes 2m\eta) = \mathbb{C}.
\end{eqnarray*} Therefore $\kappa(X,L)=0$.

Moreover $K_X+L+\rho=0$, hence $\kappa(X,K_X+L+\rho) = 0$.
On the contrary, $\kappa(X,K_X+L) = - \infty$, because, for any $m \geq 1$, $H^0(X,-m\rho)=H^0(C,m\eta)=0$.

We can now produce examples in every dimension, building them up inductively. Let $X$ be a smooth projective variety of dimension $n$, $L$ a nef line bundle on $X$ with $\kappa(X,L)=n-2$ and $\rho$ a numerically trivial divisor on $X$ such that $\kappa(X,K_X+L+\rho) \geq 0$ but $\kappa(X,K_X+L)=-\infty$. Again, let $C$ be a smooth elliptic curve. Take  $X \times C$. Call $\pi_1$ the first projection and $\pi_2$ the second one.
Fix any point $q$ on $C$ and define $\rho':=\pi_1^*(\rho)$, $L':=\pi_1^*(L) + \pi_2^*(q)$. It is clear that $L'$ is nef and that $K_{X \times C} = \pi_1^*(K_X)$. Hence, by Kunneth's formula, it is easy to see that, for $m$ sufficiently large and divisible, $$H^0(X \times C, m(K_{X \times C} +L' + \rho' ))\not = 0,$$ while, for every $m \geq 1$, $$H^0(X \times C,  m(K_{X \times C} +L'))=0.$$
Since $$\kappa(X\times C, L')= \lim_{m \rightarrow + \infty} \frac{\log(h^0(X \times C,mL'))}{\log(m)},$$ where the limit is taken over sufficiently divisible $m$'s (see for example \cite[Corollary 2.1.38]{LazI}), then it is straightforward to see that $\kappa(X \times C, L') = (n+1)-2$.
\end{example}

\bibliographystyle{plain}  
\bibliography{luoghibase}

\end{document}